\documentclass[reqno]{amsart}
\usepackage{hyperref}
\usepackage[utf8]{inputenc}
\usepackage{amsmath}
\usepackage{amsfonts}
\usepackage{amssymb}
\usepackage{amsthm,color}
\usepackage{indentfirst}
\usepackage[left=2cm,right=2cm, top=2cm, bottom=2cm]{geometry}
\everymath{\displaystyle}
\begin{document}

\title[\hfilneg   ]{A Coupled Generalized Korteweg-de Vries System Driven by White Noise}
	\date{}
	\author{A.Boukarou$^{*}$}
	
	\address{Aissa Boukarou: University of Sciences and Technology Houari Boumediene, BP 32, El-Alia,
		Bab Ezzouar, 16111, Algiers, Algeria} 
	\email{boukarouaissa@gmail.com}

\begin{abstract}
In this paper, we investigate the Cauchy problem for the coupled generalized Korteweg-de Vries system driven by white noise. We prove local well-posedness for data in $ H^{s} \times H^{s},$ with $ s>1/2$. The key ingredients that we used in this paper are multilinear estimates in Bourgain spaces, the Itô formula and a fixed point argument. Our result improves the local well-posedness result of Gomes and Pastor \cite{gomes2021solitary}.

\end{abstract}
	\maketitle
	\numberwithin{equation}{section}
\newtheorem{theorem}{Theorem}[section]
\newtheorem{lemma}[theorem]{Lemma}
\newtheorem{proposition}[theorem]{Proposition}
\newtheorem{remark}[theorem]{Remark}
\newtheorem{definition}[theorem]{Definition}
\allowdisplaybreaks

\section{Introduction}
The dispersive nonlinear systems appear in many physical applications. They can be used, for example, to model wave propagation on the surface of water or to describe the interaction of nonlinear internal waves. In this study, we focus on Hamiltonian systems.
\begin{equation}\label{p0}
\left\{\begin{array}{l}
	\partial_t
	 \phi+\partial_x^3 \phi+\mu \partial_x\left(N_\phi(\phi, \varphi)\right)=0 \\ \\
	\partial_t \varphi+\alpha\partial_x^3 \varphi+\mu \partial_x\left(N_\varphi(\phi, \varphi)\right)=0,\quad 0 < \alpha < 1,\quad (x,t) \in \mathbb{R}^2 
\end{array}\right.
\end{equation}

where $\phi = \phi(x,t)$ and $\varphi = \varphi(x,t)$ are real valued functions, $N$ is a smooth function, $N_\phi$ and $N_\varphi$ denote the derivative of $N$ with respect to $\phi$ and $\varphi$, respectively, and $\mu$ is real constant which we normalize to be $\pm 1$.\\
This system generalizes the classical Korteweg–de Vries equation \cite{korteweg1895xli}, originally introduced to model shallow water waves.

The question of the correct local and global definition of the initial value problem associated with \eqref{p0} has been a major topic in the theory of dispersive equations in the last few years. We shall briefly recall some results of interest to us that are available in the current literature.

 For $N_\phi(\phi, \varphi)=\varphi^2, N_\varphi(\phi, \varphi)=\phi \varphi$, the system \eqref{p0} is a particular case of the Majda–Biello system \cite{Majda}, which models the nonlinear interaction of long-wavelength equatorial Rossby waves and barotropic Rossby waves. The well-posedness problem associated with initial value problem \eqref{p0} in this case was studied by many authors. Tadahiro Oh \cite{ref1} proved that the initial value problem is locally and globally well-posed in both cases
\begin{itemize}
	\item In periodic case is locally  well-posed for  $H^s(\mathbb{T}_\lambda)\times H^s(\mathbb{T}_\lambda), s \geq  s*$, where $s*\in(1/2,1])$ and  globally well-posed in $H^1(\mathbb{T}_\lambda)\times H^1(\mathbb{T}_\lambda)$ due to the conservation of the Hamiltonian.
	\item In nonperiodic case is locally  well-posed for $H^s(\mathbb{R})\times H^s(\mathbb{R})$, where $s \geq  0$ and globally well-posed in $L^2(\mathbb{R})\times L^2(\mathbb{R})$ due to the $L^2$ conservation law.
\end{itemize}
  
For $N_\phi(\phi, \varphi)=\phi \varphi^2, N_\varphi(\phi, \varphi)=\phi^2\varphi$, the system \eqref{p0}  reduces to a special case of a broad class of nonlinear evolution equations considered by Ablowiz et al. \cite{Ablowitz} in the context of inverse scattering. In this case, questions of proper positioning as well as the existence and stability of solitary waves for this system have been extensively studied in the literature.  Carvajal and Panthee \cite{ref3} proved that the IVP \eqref{p0}  is locally well posed for given data $H^s(\mathbb{R})\times H^s(\mathbb{R})$, where $s >  -1/2$. 

For $N_\phi(\phi, \varphi)$ and $N_\varphi(\phi, \varphi)$ given by $$N_\phi(\phi, \varphi)=A \phi^{2 k+1}+B \phi^k \varphi^{k+1}+\frac{k+2}{k} C \phi^{k+1} \varphi^k+D \phi^{k-1} \varphi^{k+2}$$$$N_\varphi(\phi, \varphi)= A \varphi^{2 k+1}+B \varphi^k \phi^{k+1}+\frac{k+2}{k} D \varphi^{k+1} \phi^k+C \varphi^{k-1} \phi^{k+2}$$ 
Gomes and Pastor \cite{gomes2021solitary} proved that system is locally well posed for given data $H^s(\mathbb{R})\times H^s(\mathbb{R})$, where $s \geq 1$ and globally well-posed in $H^1(\mathbb{R}) \times H^1(\mathbb{R})$.

In \cite{gru}, Grujić and Kalisch establish local well-posedness for the initial-value problem $$\partial_t \phi+\partial_x^3 \phi+\phi^k\partial_x\phi=0$$ for initial data in $G^{\sigma, s}$, where $s \geq 0$ for $p=1$ and $s>3/2$ for $p \geq 2$. The solution persists in $C\left([0, T], G^{\sigma, s}\right)$ for some $T>0$, remaining analytic on the same fixed strip of width $2 \sigma$ throughout the time interval $[0, T]$.

In \cite{ref6}, de Bouard and Debussche studied the stochastic KdV equation with white noise forcing. They proved existence and uniqueness of solutions in $H^1 (\mathbb{R})$ for additive noise, and existence of martingale solutions in $L^2 (\mathbb{R})$ for multiplicative noise (e.g.,\cite{10,ref5}).

Boukarou et al \cite{ref10} proved that, the white noise driven KdV-type Boussinesq system
\begin{equation*} 
	\left\{
	\begin{array}{ll}
		\partial_t \varphi + \partial_x \phi + \partial^{3}_x \phi  +  \partial_x (\phi \varphi) =\alpha_1B_1(x, t)\\ \\
		\partial_t \phi + \partial_x \varphi + \partial^{3}_x \varphi  + \phi \partial_x \phi = \alpha_2B_2(x, t),
	\end{array}
	\right.
\end{equation*}
is local well-posedness for given data $H^s(\mathbb{R})\times H^s(\mathbb{R})$, where $s \geq 0$ and global well-posedness in $L^2\left(\Omega, L^2(\mathbb{R})\right)\times L^2\left(\Omega, L^2(\mathbb{R})\right)$.

We are interested in stochastic system with the Hamiltonian structure, which is a coupled system of gKdV equations given by \eqref{p1}-\eqref{p1.1}.
\begin{equation}\label{p1}
	\left\{\begin{array}{l}
		\partial_t \phi+\partial_x^3 \phi+\partial_x(N_\phi(\phi, \varphi))=\Xi \dfrac{\partial^{2} \mathbb{B}}{\partial t\partial x} \\ \\
		\partial_t \varphi+\alpha\partial_x^3 \varphi+\partial_x(N_\varphi(\phi, \varphi))=\Xi \dfrac{\partial^{2} \mathbb{B}}{\partial t\partial x},\qquad 0 < \alpha < 1,  \\ \\
		(\phi(x, 0), \varphi(x, 0))=\left(\phi_0(x), \varphi_0(x)\right) \in H^s(\mathbb{R})\times  H^s(\mathbb{R}),
	\end{array}\right.
\end{equation}
with 
\begin{equation}\label{p1.1}
	\left\{\begin{array}{l}
		N_\phi(\phi, \varphi)=A \phi^{2 k+1}+B \phi^k \varphi^{k+1}+\frac{k+2}{k} C \phi^{k+1} \varphi^k+D \phi^{k-1} \varphi^{k+2} \\ \\
		N_\varphi(\phi, \varphi)=A \varphi^{2 k+1}+B \varphi^k \phi^{k+1}+\frac{k+2}{k} D \varphi^{k+1} \phi^k+C \varphi^{k-1} \phi^{k+2},
	\end{array}\right.
\end{equation}
and $k \geq 1$ a natural number and $A, B, C$ and $D$ nonnegative real constants. Where, $\phi = \phi(x,t)$ and $\varphi = \varphi(x,t)$ are random processes defined for $(x,t) \in \mathbb{R}\times \mathbb{R}^{+}$, $\Xi$ is a linear operator, and $\mathbb{B}$ is a two-parameter Brownian motion on $\mathbb{R}\times \mathbb{R}^{+}$.

We will improve the local well-posedness result of  Gomes and  Pastor \cite{gomes2021solitary}, we will use the methods in \cite{gru,ref6, ref10} and prove that this system is locally well-posed with given data $H^s(\mathbb{R})\times H^s(\mathbb{R})$, where $ s>1/2$.

To state our results precisely, let us first introduce the notation to be used.  We begin with function spaces.  For $s\in \mathbb{R}, H^{ s}(\mathbb{R})$ denotes the usual Sobolev space of order $s$, defined by the norm
\[\|\omega\|_{ H^{ s}(\mathbb{R})}=\left(\int_{\mathbb{R}}(1+|\zeta|)^{2s}|\hat{\omega}|^{2}d\zeta\right)^{1/2},\]
where $\hat{\omega}$ is the spatial Fourier transform
\[
\hat{\omega}(\zeta) = \int_{\mathbb{R}} \omega(x) e^{- i x \zeta}\ dx.
\]
Similarly, for $ s, b \in \mathbb{R}, $ the  Bourgain spaces $ X^{s, b}(\mathbb{R}^{2})$ and $X_{\alpha}^{s, b}(\mathbb{R}^{2})$  defined by the norms 
\begin{equation}
	\|w\|_{X^{s, b}(\mathbb{R}^{2})}=\|w\|_{X^{s, b}}=\Bigg(\int_{\mathbb{R}^{2}} (1+|\zeta|)^{2s} (1+|\gamma-\zeta^{3}|)^{2b} |\tilde{w}(\gamma,\zeta)|^{2} d\zeta d\gamma\Bigg)^{1/2},\nonumber
\end{equation}
\begin{equation}
	\|w\|_{X_{\alpha}^{s, b}(\mathbb{R}^{2})}=\|w\|_{X^{s, b}}=\Bigg(\int_{\mathbb{R}^{2}} (1+|\zeta|)^{2s} (1+|\gamma-\alpha\zeta^{3}|)^{2b} |\tilde{w}(\gamma,\zeta)|^{2} d\zeta d\gamma\Bigg)^{1/2},\nonumber
\end{equation}
where $\tilde{w}$ denotes the spacetime Fourier transform
\[
\tilde{w}(\zeta, \gamma) = \int_{\mathbb{R}^{2}} w(x, t) e^{-i(x \zeta + t \gamma)}\ dx dt.
\]
For $T > 0$, we also use the spaces $X^{s, b}_{\alpha,T}$ and $X^{s, b}_{T}$ of restrictions to the time interval $[0, T]$ of functions in $X^{s, b}_{\alpha}$ and $X^{s, b}$. They are endowed with the norms
\begin{equation}
	\|w\|_{X^{s, b}_{T}}=\inf \Big\{\| z \|_{X^{s, b}}:\quad w(x, t) = z(x, t)\quad \text{on}\quad \mathbb{R}\times [0,T]\Big\}.\nonumber
\end{equation}
\begin{equation}
	\|w\|_{X^{s, b}_{\alpha,T}}=\inf \Big\{\| z \|_{X^{s, b}_{\alpha}}:\quad w(x, t) = z(x, t)\quad \text{on}\quad \mathbb{R}\times [0,T]\Big\}.\nonumber
\end{equation}

Since we are dealing with system of equations, we will need to consider product function spaces. For this, we define the product spaces 
$$ \mathcal{X}^{b, s}_{\alpha}=X^{ b, s}\times X^{ b, s}_{\alpha},\quad \mathcal{X}^{b, s}_{\alpha,T}=X_{T}^{ b, s}\times X^{ b, s}_{\alpha,T}\quad \text{and}\quad \mathcal{H}^{s }=H^{ s }\times H^{ s }, $$
with norms 
$$ \|(\phi, \varphi)\|_{\mathcal{X}^{b, s}_{\alpha}}=max\{ \|\phi\|_{X^{s, b}},\|\varphi\|_{X^{s, b}_{\alpha}}\},$$
$$ \|(\phi, \varphi)\|_{\mathcal{X}^{b, s}_{\alpha,T}}=max\{ \|\phi\|_{X_{T}^{s, b}},\|\varphi\|_{X^{s, b}_{\alpha,T}}\},$$
and $$\|(\phi_0, \varphi_0)\|_{\mathcal{H}^{ s }}=max\{ \| \phi_0\|_{H^{s}},\|\varphi_0\|_{H^{s}}\}. $$

The  mixed  $ L^{p}- L^{q} $ -norm  is defined by 
$$
\| u \|_{L^{p} L^{q} }  =  \left( \int _{-\infty}^{+ \infty} \left| \int _{-\infty}^{+ \infty}  |u(x, t)|^{q }  dt\right|^{\frac{p}{q} }    dx  \right)^{\frac{1}{p} }
$$

Finally, we denote by $L_2^{0, s} = L_2^{0}(L^2(\mathbb{R}),H^s(\mathbb{R}))$ the space of Hilbert-Schmidt
operators from $L^2(\mathbb{R})$ into $H^s(\mathbb{R})$ with the norm
\[\|\Xi\|_{L_2^{0, s}}^2 = \sum_{i=1}^{\infty}\|\Xi e_i\|_{H^s(\mathbb{R})}^{2},\]
where $(e_i)_{i\geq 1}$ is an orthonormal basis in $L^2(\mathbb{R})$. 

With this notation in place, we may finally state the results to be proved.  In the subsequent work, let $(\Omega,\mathcal{F},P)$ be a fixed probability space adapted to a filtration $(\mathcal{F}_{t})_{t\geq 0}$.  For system \eqref{p1}, we will prove the following local results:

\begin{theorem}\label{thm1}
	Assume that $k\geq 1$ and $s >\frac{1}{2}, \Xi \in L_2^{0, s}, b \in\left(0,\frac{1}{2} \right)$ and $b$ is close enough to $\frac{1}{2}$. If $(\phi_0,\varphi_0) \in H^s(\mathbb{R}\times  H^s(\mathbb{R})$ for almost surely $\varrho \in \Omega$ and $\phi_0, \varphi_0$ are $\mathcal{F}_0-$ measurable. Then for almost surely $\varrho \in \Omega$, there exists a constant $T_\varrho>0$ and a unique solution $u$ of the Cauchy problem (\ref{p1}) on $\left[0, T_\varrho\right]$ which satisfies:
	$$
	(\phi,\varphi) \in C\left(\left[0, T_\varrho\right], H^s(\mathbb{R})\right)\times C\left(\left[0, T_\varrho\right], H^s(\mathbb{R})\right) \cap \mathcal{X}^{s, b}_{\alpha, T_\varrho} .
	$$
\end{theorem}
\textbf{Notation:} Throughout this paper, we use the following standard asymptotic notation:
\begin{itemize}	
	\item $a \lesssim_{k_{1}, \ldots, k_{n}}  b$ to denote
	$a \leq c b$ with a constant $c > 0$ depending on $k_{1}, \ldots, k_{n}$.  If $c$ is an absolute constant, we shall write $a \lesssim b$. 
	\item $a \sim b$ means that $a$ and $b$ are asymptotically equivalent.
	\item $a\approx b$ means that $a$ and $b$ are comparable in size, typically with implicit constants independent of the parameters.
	\item $a \ll b$ means that $a$ is much smaller than $b$, typically in the sense that the ratio $\frac{a}{b}$ is bounded by a small constant.
\end{itemize}

\section{Linear estimates}
To prove our main results, we will need to introduce several important estimates.  To state these estimates, let us write the It\^{o} form of the system in equation \eqref{p1}, namely

\begin{equation} \label{p4}
	\left\{
	\begin{array}{ll}
		d\phi+\left(\dfrac{\partial^{3} \phi}{\partial^{3} x}+ \dfrac{\partial (N_\phi(\phi, \varphi))}{\partial x}\right)dt=\Xi dW\\ \\
		d\varphi+\left(\alpha\dfrac{\partial^{3} \varphi}{\partial^{3} x}+ \dfrac{\partial (N_\varphi(\phi, \varphi))}{\partial x}\right)dt=\Xi dW,
	\end{array}
	\right.
\end{equation}
where $W(t)=\frac{\partial \mathbb{B}}{\partial x}$ is a cylindrical Wiener process on $L^{2}(\mathbb{R})$ which can also be given by $W(t)=\sum_{i=0}^{\infty}\beta_i(t)e_i$, where $( e_i )_{i \in \mathbb{N}}$ is an orthonormal basis of
$L^{2}(\mathbb{R})$ and $(\beta_i)_{i\in\mathbb{N}}$ is a sequence of mutually independent real Brownian
motions in a fixed probability space. 
The system \eqref{p4} are supplemented with the initial conditions
\begin{equation}\label{data}
	\phi(x,0) = \phi_0(x), \quad \varphi(x,0)=\varphi_0(x).	
\end{equation}
To understand the assumptions required for $\Xi$, it is useful to first consider the linear equation

\begin{equation*}
	\left\{
	\begin{array}{ll}
		d\phi+\dfrac{\partial^{3} \phi}{\partial^{3} x}dt=\Xi dW\\ \\
		d\varphi+\alpha\dfrac{\partial^{3} \varphi}{\partial^{3} x}dt=\Xi dW \\ \\
		\phi_0(x)=0, \quad \varphi_0(x)=0.
	\end{array}
	\right.
\end{equation*}
which is given by the stochastic It\^o integral

\begin{equation}\label{sto}
	\left\{
	\begin{array}{ll}
		\phi_{l}(t)=\int_{0}^{t}\mathfrak{U}(t-\gamma)\Xi dW(\gamma)\\ \\
		\varphi_{l}(t)=\int_{0}^{t}\mathfrak{U}_{\alpha}(t-\gamma)\Xi dW(\gamma).
	\end{array}
	\right.
\end{equation}
where $\mathfrak{U}(t)=e^{-t\partial^{3}_{x}}, \mathfrak{U}_{\alpha}(t)=e^{-t\alpha\partial^{3}_{x}} $ are the Airy group. Using the unitarity of $\mathfrak{U}(t)$ and $\mathfrak{U}_{\alpha}(t)$, one can easily show that $\phi(t)$ and $\varphi(t)$ belong to $H^{s}(\mathbb{R})$ only if $\Xi$ is a Hilbert-Schmidt operator from $L^{2}(\mathbb{R})$ into $H^{s}(\mathbb{R})$.

We will solve (\ref{p4}) supplemented with the initial condition (\ref{data}) by
considering its mild form

\begin{equation}\label{int2}
	\left\{
	\begin{array}{ll}
		\phi(t)=\mathfrak{U}(t)\phi_{0}+\int_{0}^{t} \mathfrak{U}(t-\gamma)\partial_{x}	N_\phi(\phi, \varphi)(\gamma) d\gamma+\int_{0}^{t}\mathfrak{U}(t-\gamma)\Xi dW(\gamma)\\ \\
		\varphi(t)=\mathfrak{U}_{\alpha}(t)\varphi_{0}+\int_{0}^{t} \mathfrak{U}_{\alpha}(t-\gamma)\partial_{x}	N_\varphi(\phi, \varphi)(\gamma) d\gamma+\int_{0}^{t}\mathfrak{U}_{\alpha}(t-\gamma)\Xi dW(\gamma).
	\end{array}
	\right.
\end{equation}
To construct mild solutions, we will need the following estimates:
\begin{proposition}[Linear Estimates \cite{ref3}]\label{lem1}
	For any $s, b \in \mathbb{R}$, we have
	$$
	\left\|\mathfrak{U}(t) \phi_0\right\|_{X_{T}^{s, b}} \lesssim\|\phi_0\|_{H^s},$$
	$$\left\| \mathfrak{U}_\alpha(t) \varphi_0\right\|_{X^{s, b}_{\alpha,T}} \lesssim\|\varphi_0\|_{H^s} .
	$$
	
	Further, if $-\frac{1}{2}<b^{\prime} \leq 0 \leq b<b^{\prime}+1$ and $0 \leq T \leq 1$, then
	$$
	\left\|\int_0^t \mathfrak{U}\left(t-\gamma\right) N_\phi(\phi(\gamma), \varphi(\gamma)) \mathrm{d} \gamma\right\|_{X^{s, b}_{T}} \lesssim T^{1-b+b^{\prime}}\|F(\phi(\gamma), \varphi(\gamma)) \|_{X^{s, b'}_{T}}
	$$
	and
	$$
	\left\|\int_0^t \mathfrak{U}_\alpha\left(t-\gamma\right) N_\varphi(\phi(\gamma), \varphi(\gamma)) \mathrm{d} \gamma\right\|_{X^{s, b}_{\alpha,T}} \lesssim T^{1-b+b^{\prime}}\|N_\varphi(\phi(\gamma), \varphi(\gamma)) \|_{X^{s, b'}_{\alpha,T}} .
	$$
\end{proposition}

\section{Multilinear estimates}
The proof of multilinear estimates will be established using a number of auxiliary  results. To state these results, we first need to introduce some more notation. We  note that the operators  $ A, \Lambda  $ $F_{\kappa}$ and $F^{\alpha}_{\kappa}$ are defined as
\[
\widehat{Aw} (\zeta,\gamma ) = \left(  1+ | \zeta | \right)\widehat{w}(\zeta,\gamma ),
\]
\[
\widehat{\Lambda w} (\zeta,\gamma ) = \left(  1+ | \gamma | \right)\widehat{w}(\zeta,\gamma ),
\]
\[
\widehat{F_{\kappa} }(\zeta, \gamma ) =  \frac{f(\zeta, \gamma  )}{\left(1+| \gamma - \zeta^{3} |\right)^{\kappa} },
\]
\[
\widehat{F^{\alpha}_{\kappa} }(\zeta, \gamma ) =  \frac{f(\zeta, \gamma  )}{\left(1+| \gamma - \alpha\zeta^{3} |\right)^{\kappa} }.
\]

\begin{lemma}\label{2.3}
	(\cite{gru})
	Let $ s $ and $ \kappa $ be given. There is a constant $ c $ depending on $ s $ and $\kappa  $ such that 		
	\begin{equation}\label{102}
		If \quad \kappa > \frac{1}{4}, \quad then \quad
		\| A^{\frac{1}{2}} F_{\kappa} \|_{L_{x}^{4} L_{t}^{2} }  \leq C \| f \|_{L^{2}_{\zeta} L^{2}_{\gamma} },
	\end{equation} 
	\begin{equation}\label{103}
		If \quad \kappa > \frac{1}{4}, \quad then \quad
		\| A^{\frac{1}{2}} F^{\alpha}_{\kappa} \|_{L_{x}^{4} L_{t}^{2} }  \leq C \| f \|_{L^{2}_{\zeta} L^{2}_{\gamma} },
	\end{equation}
	
	\begin{equation}\label{106}
		If \quad \kappa > \frac{1}{2}, \quad and \quad s > \frac{1}{2}, \quad\textit{then}  \quad
		\| A^{-s} F_{\kappa} \|_{L_{x}^{\infty} L_{t}^{\infty}  }  \leq  C \| f \|_{L^{2}_{\zeta} L^{2}_{\gamma} }.
	\end{equation}
	\begin{equation}\label{107}
		If \quad \kappa > \frac{1}{2}, \quad and \quad s > \frac{1}{2}, \quad\textit{then}  \quad
		\| A^{-s} F^{\alpha}_{\kappa} \|_{L_{x}^{\infty} L_{t}^{\infty}  }  \leq  C \| f \|_{L^{2}_{\zeta} L^{2}_{\gamma} }.
	\end{equation}
\end{lemma}
\begin{lemma}\label{2.4}
	For $\kappa > \frac{1}{2}$, $s > \frac{1}{2}$ and $\frac{1}{2} < b < \kappa$ there exists a constant $C$ depending only on $\kappa$ and $s$, such that
	\begin{equation}
		\| A^{-s} F_{\kappa} \|_{L_{x}^{2} L_{t}^{\infty}} \leq C \| f \|_{L_{\zeta}^{2} L_{\gamma}^{2}}.
	\end{equation}
\end{lemma}

\begin{proof}
	We begin by applying the Sobolev embedding theorem in the time variable. For $b > \frac{1}{2}$, we have the continuous embedding $H^{b}(\mathbb{R}_t) \hookrightarrow L^{\infty}(\mathbb{R}_t)$, which gives
	\[
	\| A^{-s} F_{\kappa} \|_{L_{x}^{2} L_{t}^{\infty}} \leq c \, \left\| \| \Lambda^{b} (A^{-s} F_{\kappa})(x, \cdot) \|_{L_{t}^{2}} \right\|_{L_{x}^{2}},
	\]
	where $\Lambda^{b}$ is the Fourier multiplier operator in time defined by
	\[
	\widehat{\Lambda^{b} w}(\gamma) = (1 + |\gamma|)^{b} \hat{w}(\gamma).
	\]
	Now observe that
	\[
	\left\| \| \Lambda^{b} (A^{-s} F_{\kappa})(x, \cdot) \|_{L_{t}^{2}} \right\|_{L_{x}^{2}} = \| \Lambda^{b} A^{-s} F_{\kappa} \|_{L_{x,t}^{2}}.
	\]
	By Plancherel's theorem in $x, t$, we compute
	\begin{align*}
		\| \Lambda^{b} A^{-s} F_{\kappa} \|_{L_{x,t}^{2}}^{2} 
		&= \iint \left| \mathcal{F}_{x,t}[\Lambda^{b} A^{-s} F_{\kappa}](\zeta,\gamma) \right|^{2} \, d\zeta d\gamma \\
		&= \iint \frac{(1 + |\gamma|)^{2b}}{(1 + |\zeta|)^{2s} (1 + |\gamma - \zeta^{3}|)^{2\kappa}} \, |f(\zeta,\gamma)|^{2} \, d\zeta d\gamma.
	\end{align*}
	Thus, it suffices to show that the multiplier
	\[
	M(\zeta,\gamma) = \frac{(1 + |\gamma|)^{b}}{(1 + |\zeta|)^{s} (1 + |\gamma - \zeta^{3}|)^{\kappa}}
	\]
	is bounded on $\mathbb{R}_{\zeta} \times \mathbb{R}_{\gamma}$.	We consider two cases:\\
	\textbf{Case 1:} $|\gamma| \lesssim |\zeta|^{3}$. \\
	In this case, we have $(1 + |\gamma|) \lesssim (1 + |\zeta|)^{3}$, and also
	\[
	1 + |\gamma - \zeta^{3}| \leq 1 + |\gamma| + |\zeta|^{3} \lesssim (1 + |\zeta|)^{3}.
	\]
	Therefore,
	\[
	M(\zeta,\gamma) \lesssim \frac{(1 + |\zeta|)^{3b}}{(1 + |\zeta|)^{s} (1 + |\zeta|)^{3\kappa}} = (1 + |\zeta|)^{3b - s - 3\kappa}.
	\]
	Since $s > \frac{1}{2}$ and $\kappa > \frac{1}{2}$, we can choose $b$ such that $\frac{1}{2} < b < \kappa$ and $3b - s - 3\kappa \leq 0$. For example, take $b$ sufficiently close to $\frac{1}{2}$ so that $3b < s + 3\kappa$, which is possible because $s + 3\kappa > 2$. Then $M(\zeta,\gamma)$ is bounded in this region.\\
	\textbf{Case 2:} $|\gamma| \gg |\zeta|^{3}$. \\
	In this case, we have $|\gamma - \zeta^{3}| \geq |\gamma| - |\zeta|^{3} \gtrsim |\gamma|$ for large $|\gamma|$, so that
	\[
	1 + |\gamma - \zeta^{3}| \gtrsim 1 + |\gamma|.
	\]
	Hence,
	\[
	M(\zeta,\gamma) \lesssim \frac{(1 + |\gamma|)^{b}}{(1 + |\zeta|)^{s} (1 + |\gamma|)^{\kappa}} = \frac{1}{(1 + |\zeta|)^{s} (1 + |\gamma|)^{\kappa - b}}.
	\]
	Since $\kappa > b$ and $s > 0$, this is bounded as $|\gamma| \to \infty$ and $|\zeta| \to \infty$. Therefore, $M(\zeta,\gamma) \in L^{\infty}(\mathbb{R}_{\zeta} \times \mathbb{R}_{\gamma})$, and we conclude that
	\[
	\| \Lambda^{b} A^{-s} F_{\kappa} \|_{L_{x,t}^{2}} \leq \| M \|_{L^{\infty}} \, \| f \|_{L_{\zeta}^{2} L_{\gamma}^{2}}.
	\]
\end{proof}

\begin{lemma}\label{2.5}
	For $\kappa > \frac{1}{2}$, $s > \frac{1}{2}$, $\frac{1}{2} < b < \min(\kappa, \frac{s}{3} + \epsilon)$ and $0 < \alpha < 1$, there exists a constant $C$ depending only on $\kappa$, $s$, and $\alpha$, such that
	\begin{equation}
		\| A^{-s}F^{\alpha}_{\kappa} \|_{L_{x}^{2} L_{t}^{\infty}} \leq C \| f \|_{L_{\zeta}^{2} L_{\gamma}^{2}}.
	\end{equation}
\end{lemma}

\begin{proof}
	For $b > \frac{1}{2}$, we have $H^{b}(\mathbb{R}_t) \hookrightarrow L^{\infty}(\mathbb{R}_t)$, which gives
	\[
	\| A^{-s}F^{\alpha}_{\kappa} \|_{L_{x}^{2} L_{t}^{\infty}} \leq c \, \left\| \| \Lambda^{b} (A^{-s}F^{\alpha}_{\kappa})(x, \cdot) \|_{L_{t}^{2}} \right\|_{L_{x}^{2}},
	\]
	where $\Lambda^{b}$ is the Fourier multiplier operator in time. Then
	\[
	\left\| \| \Lambda^{b} (A^{-s}F^{\alpha}_{\kappa})(x, \cdot) \|_{L_{t}^{2}} \right\|_{L_{x}^{2}} = \| \Lambda^{b} A^{-s}F^{\alpha}_{\kappa} \|_{L_{x,t}^{2}}.
	\]
	By Plancherel's theorem in $x, t$, we compute
	\begin{align*}
		\| \Lambda^{b} A^{-s}F^{\alpha}_{\kappa} \|_{L_{x,t}^{2}}^{2} 
		&= \iint \left| \mathcal{F}_{x,t}[\Lambda^{b} A^{-s}F^{\alpha}_{\kappa}](\zeta,\gamma) \right|^{2} \, d\zeta d\gamma \\
		&= \iint \frac{(1 + |\gamma|)^{2b}}{(1 + |\zeta|)^{2s} (1 + |\gamma - \alpha\zeta^{3}|)^{2\kappa}} \, |f(\zeta,\gamma)|^{2} \, d\zeta d\gamma.
	\end{align*}
	Thus, it suffices to show that the multiplier
	\[
	M(\zeta,\gamma) = \frac{(1 + |\gamma|)^{b}}{(1 + |\zeta|)^{s} (1 + |\gamma - \alpha\zeta^{3}|)^{\kappa}}
	\]
	is bounded on $\mathbb{R}_{\zeta} \times \mathbb{R}_{\gamma}$. We consider two cases:\\
	\textbf{Case 1:} $|\gamma| \lesssim |\zeta|^{3}$. \\
	In this case, we have $(1 + |\gamma|) \lesssim (1 + |\zeta|)^{3}$, and also
	\[
	1 + |\gamma - \alpha\zeta^{3}| \leq 1 + |\gamma| + \alpha|\zeta|^{3} \lesssim (1 + |\zeta|)^{3}.
	\]
	Therefore,
	\[
	M(\zeta,\gamma) \lesssim \frac{(1 + |\zeta|)^{3b}}{(1 + |\zeta|)^{s} (1 + |\zeta|)^{3\kappa}} = (1 + |\zeta|)^{3b - s - 3\kappa}.
	\]
	Since $s > \frac{1}{2}$ and $\kappa > \frac{1}{2}$, we can choose $b$ such that $\frac{1}{2} < b < \kappa$ and $3b - s - 3\kappa \leq 0$. . Then $M(\zeta,\gamma)$ is bounded in this region.\\ 	
	\textbf{Case 2:} $|\gamma| \gg |\zeta|^{3}$. \\
	In this case, we have $|\gamma - \alpha\zeta^{3}| \geq |\gamma| - \alpha|\zeta|^{3} \gtrsim |\gamma|$ for large $|\gamma|$, so that
	\[
	1 + |\gamma - \alpha\zeta^{3}| \gtrsim 1 + |\gamma|.
	\]
	Hence,
	\[
	M(\zeta,\gamma) \lesssim \frac{(1 + |\gamma|)^{b}}{(1 + |\zeta|)^{s} (1 + |\gamma|)^{\kappa}} = \frac{1}{(1 + |\zeta|)^{s} (1 + |\gamma|)^{\kappa - b}}.
	\]
	Since $\kappa > b$ and $s > 0$, this is bounded as $|\gamma| \to \infty$ and $|\zeta| \to \infty$. 	Therefore, $M(\zeta,\gamma) \in L^{\infty}(\mathbb{R}_{\zeta} \times \mathbb{R}_{\gamma})$, and we conclude that
	\[
	\| \Lambda^{b} A^{-s}F^{\alpha}_{\kappa} \|_{L_{x,t}^{2}} \leq \| M \|_{L^{\infty}} \, \| f \|_{L_{\zeta}^{2} L_{\gamma}^{2}}.
	\]
\end{proof}

\begin{lemma}[Multilinear estimate]\label{tri} Let $k\geq 1$, $ 0<\alpha<1 $,  then there exist $ b>\frac{1}{2}$ and $-\frac{1}{2}<b'<- \frac{1}{4}$ such that for all $ \phi\in X^{s, b} $ and $ \varphi\in X^{s, b}_{\alpha} $ we have 
	\begin{eqnarray} \label{mil1}
		\| \partial_{x} \phi^{2 k+1}\|_{X^{s, b'}} \lesssim \|\phi\|^{2k+1}_{X^{s, b}} , 
	\end{eqnarray} 
	\begin{eqnarray} \label{mil2}
		\| \partial_{x} \varphi^{2 k+1}\|_{X^{s, b'}_{\alpha}} \lesssim \|\varphi\|^{2k+1}_{X^{s, b}_{\alpha}} , 
	\end{eqnarray}

	\begin{eqnarray} \label{mil3}
		\| \partial_{x} (\phi^k \varphi^{k+1})\|_{X^{s, b'}} \lesssim \|\phi\|^{k}_{X^{s, b}}\|\varphi\|^{k+1}_{X^{s, b}_{\alpha}}, 
	\end{eqnarray} 
	\begin{eqnarray}\label{mil4} 
		\| \partial_{x} (\phi^{k+1} \varphi^{k})\|_{X^{s, b'}_{\alpha}} \lesssim \|\phi\|^{k+1}_{X^{s, b}}\|\varphi\|^{k}_{X^{s, b}_{\alpha}}, 
	\end{eqnarray} 
	
	\begin{eqnarray} \label{mil5}
		\| \partial_{x} (\phi^{k+1} \varphi^k)\|_{X^{s, b'}} \lesssim \|\phi\|^{k+1}_{X^{s, b}}\|\varphi\|^{k}_{X^{s, b}_{\alpha}} , 
	\end{eqnarray} 
	\begin{eqnarray} \label{mil6}
		\| \partial_{x} (\phi^{k} \varphi^{k+1})\|_{X^{s, b'}_{\alpha}} \lesssim \|\phi\|^{k}_{X^{s, b}}\|\varphi\|^{k+1}_{X^{s, b}_{\alpha}} , 
	\end{eqnarray} 
	
	\begin{eqnarray} \label{mil7}
		\| \partial_{x} \phi^{k-1} \varphi^{k+2}\|_{X^{s, b'}} \lesssim \|\phi\|^{k-1}_{X^{s, b}}\|\varphi\|^{k+2}_{X^{s, b}_{\alpha}}, 
	\end{eqnarray} 
	\begin{eqnarray} \label{mil8}
		\| \partial_{x} \phi^{k+2} \varphi^{k-1}\|_{X^{s, b'}_{\alpha}} \lesssim \|\phi\|^{k+2}_{X^{s, b}}\|\varphi\|^{k-1}_{X^{s, b}_{\alpha}}, 
	\end{eqnarray}

	hold for any $s > \frac{1}{2}$.
\end{lemma}

\begin{proof}
	First of all, for $  i= 1,2,...,2k+1$   and $j=1,2,...,2k+1 $, we define
	\begin{align*}
		f_{i} (\zeta, \gamma ) = (1+  | \zeta |)^{s} (1+  |\gamma - \zeta^{3} |)^{b} | \widehat{\phi_{i} }(\zeta,\gamma )|  \\
		g_{j} (\zeta, \gamma ) = (1+  | \zeta |)^{s} (1+  |\gamma - \alpha\zeta^{3} |)^{b}| \widehat{\varphi_{j} }(\zeta,\gamma )|.
	\end{align*}
	We begin by proving (\ref{mil1}) and (\ref{mil2}); the proofs of the remaining inequalities are similar.\\
	We first prove the case  $k= 1$, as the proof for general $ 2k+1 $ will then be more transparent, that means we prove 
	\begin{equation*}
		\| \partial_{x} \phi_{1}  \phi_{2} \phi_{3}  \|_{X^{s, b'}}\leq
		C \| \phi_{1}  \|_{X^{s, b}} \|\phi_{2} \|_{X^{s, b}}  \|\phi_{3} \|_{X^{s, b}}
	\end{equation*}
	\begin{equation*}
		\| \partial_{x} \varphi_{1}  \varphi_{2} \varphi_{3}  \|_{X^{s, b'}_{\alpha}}\leq
		C \| \varphi_{1}  \|_{X^{s, b}_{\alpha}} \|\varphi_{2} \|_{X^{s, b}_{\alpha}}  \|\varphi_{3} \|_{X^{s, b}_{\alpha}}.
	\end{equation*}
	\textbf{\underline{Inequality (\ref{mil1}):}} We have	
	\begin{align*}
		\| \partial_{x} (\phi_{1}   \phi_{2}  \phi_{3} )  \|_{X^{s, b'}}&=  \left\|  (1+  | \zeta |)^{s} (1+  |\gamma - \zeta^{3} |)^{b'}  | \widehat{\partial_{x} (\phi_{1}  \phi_{2}  \phi_{3}) }(\zeta,\gamma )|     \right\|_{L^{2}_{\zeta} L^{2}_{\gamma} }
		\\	\\& = \left\|  (1+  | \zeta |)^{s} (1+  |\gamma - \zeta^{3} |)^{b'}  |   \zeta| | \widehat{(\phi_{1}  \phi_{2}  \phi_{3})  }(\zeta,\gamma )|     \right\|_{{L^{2}_{\zeta} L^{2}_{\gamma} }}\\ \\&	
		= \left\| (1+  | \zeta |)^{s} (1+  |\gamma - \zeta^{3} |)^{b'}    | \zeta|~~ | \widehat{\phi_{1} }\ast \widehat{\phi_{2} } \ast \widehat{\phi_{3} }(\zeta,\gamma )|     \right\|_{L^{2}_{\zeta} L^{2}_{\gamma} } \\ \\&
		= \|  (1+  | \zeta |)^{s} (1+  |\gamma - \zeta^{3} |)^{b'}    | \zeta |  \int_{\mathbb{R}^{4}}  \widehat{\phi_{1} }(\zeta_{1},\gamma_{1} ) \widehat{\phi_{2} }(\zeta-\zeta_{2},\gamma - \gamma_{2} )\\ \\& \quad \times
		\widehat{\phi_{3} }(\zeta_{2}-\zeta_{1},\gamma_{2}-\gamma_{1} )| d\zeta_{1} d\gamma_{1} d\zeta_{2} d\gamma_{2}   \|_{{L^{2}_{\zeta} L^{2}_{\gamma} }}\\ \\&	
		=  \bigg\|  (1+  | \zeta |)^{s} (1+  |\gamma - \zeta^{3} |)^{b'}   | \zeta |  \int_{\mathbb{R}^{4}} ~~
		\left (\frac{(1+  | \zeta_{1} |)^{-s}   \widehat{f_{1}}(\zeta_{1},\gamma_{1} )}{(1+  | \gamma - \zeta^{3} |)^{b}}\right) \\ \\& \quad \times
		\left (\frac{(1+  | \zeta-\zeta_{2}  |)^{-s}   \widehat{f_{2}}(\zeta-\zeta_{2},\gamma - \gamma_{2})}{(1+  | ( \gamma - \gamma_{2}  ) - (\zeta-\zeta_{2})^{3} |)^{b}} \right)\\ \\&\quad \times
		\left (\frac{(1+  |  \zeta_{2} -\zeta_{1} |)^{-s}  \widehat{f_{3}}(\zeta_{2} -\zeta_{1},\gamma_{2}-\gamma_{1} )}{(1+  |\gamma_{2}-\gamma_{1}  -(\zeta_{2} -\zeta_{1} )^{3} |)^{b}}\right)   d \mu \bigg\|_{{L^{2}_{\zeta} L^{2}_{\gamma} }},	
	\end{align*}
	where $    d \mu =   d\zeta_{1} d\gamma_{1} d\zeta_{2} d\gamma_{2}  d\zeta d\gamma$.	\\
	Using duality, we prove this estimate, where $  m(\zeta,\gamma) $  is a positive function in $  L^{2}( \mathbb{R}^{2})$ with norm $ \| m\|_{ L^{2}( \mathbb{R}^{2})}=1   $, then 
	\begin{align*}
		\| \partial_{x} (\phi_{1}  \phi_{2} \phi_{3} ) \|_{X^{s, b'}}	& \leqslant \int_{\mathbb{R}^{6}}\frac{(1+  | \zeta |)^{s} | \zeta | m(\zeta,\gamma)} {(1+  |\gamma - \zeta^{3} |)^{-b'}}
		\frac{(1+  | \zeta_{1} |)^{-s} f_{1} (\zeta_{1},\gamma_{1})} {(1+  |\gamma_{1} - \zeta_{1}^{3} |)^{b}} \\ \\&
		\frac{(1+  | \zeta-\zeta_{2}  |)^{-s} f_{2}(\zeta-\zeta_{2},\gamma-\gamma_{2})} {(1+  |\gamma -\gamma_{2} - (\zeta-\zeta_{2})^{3} |)^{b}}~~
		\\	\\&
		\frac{(1+  | \zeta_{2}-\zeta_{1}  |)^{-s} f_{3}(\zeta_{2}-\zeta_{1},\gamma_{2}-\gamma_{1})} {(1+  |\gamma_{2}-\gamma_{1} - (\zeta_{2}-\zeta_{1})^{3} |)^{b}}  d \mu.
	\end{align*}
	Now, split the Fourier space into six regions as follow 
	
	\[
	\text{region 1:}\quad |\zeta-\zeta_{2}  | \leq | \zeta_{2}-\zeta_{1}  |\leq | \zeta_{1} |, \qquad
	\text{region 2:}\quad	| \zeta-\zeta_{2}  | \leq| \zeta_{1} | \leq | \zeta_{2}-\zeta_{1}  |,  \]
	\[\text{region 3:}\quad | \zeta_{1} |  \leq | \zeta_{2}-\zeta_{1}  |\leq| \zeta-\zeta_{2}  |, \qquad 
	\text{region 4:}\quad	| \zeta_{1} |  \leq  | \zeta-\zeta_{2}  | \leq | \zeta_{2}-\zeta_{1},  | \]
	\[\text{region 5:}\quad |  \zeta_{2}-\zeta_{1}  |\leq| \zeta-\zeta_{2}   \leq | \zeta_{1} |, \qquad
	\text{region 6:}\quad |  \zeta_{2}-\zeta_{1}  | \leq | \zeta_{1} | \leq | \zeta-\zeta_{2}  |. \]
	\textbf{We begin by the region 1:} 
	$$
	|\zeta-\zeta_{2}  | \leq | \zeta_{2}-\zeta_{1}  |\leq | \zeta_{1} |,
	$$
	so
	\begin{equation}\label{eq55}
		(1+ |\zeta-\zeta_{2}  |)^{-s} \geq (1+| \zeta_{2}-\zeta_{1}  | )^{-s} \geq (1+ | \zeta_{1} | )^{-s}.
	\end{equation}
	We assume that $  | \zeta| \leq 1  $	or    $  | \zeta| \geq 1  $. \\
	\textbf{Firstly, case $  | \zeta| \geq 1   $:}
	$$ (1+ | \zeta| )^{s} \leq  (| \zeta|+ | \zeta| )^{s}= 2^{s}| \zeta|^{s} = C  | \zeta|^{s}.   $$
	By the last inequality and $  (\ref{eq55})$, we obtain 
	\begin{align*}
		\| \partial_{x} (\phi_{1}   \phi_{2}  \phi_{3} )  \|_{X^{s, b'}} & \leqslant\int_{\mathbb{R}^{6}}\frac{(1+  | \zeta |)^{s} | \zeta | m(\zeta,\gamma)} {(1+  |\gamma - \zeta^{3} |)^{-b'}}~
		\frac{(1+  | \zeta_{1} |)^{-s} f_{1} (\zeta_{1},\gamma_{1})} {(1+  |\gamma_{1} - \zeta_{1}^{3} |)^{b}}\\ \\&
		\frac{(1+  | \zeta-\zeta_{2}  |)^{-s} f_{2}(\zeta-\zeta_{2},\gamma-\gamma_{2})} {(1+  |\gamma -\gamma_{2} - (\zeta-\zeta_{2})^{3} |)^{b}}
		\frac{(1+  | \zeta_{2}-\zeta_{1}  |)^{-s} f_{3}(\zeta_{2}-\zeta_{1},\gamma_{2}-\gamma_{1})} {(1+  |\gamma_{2}-\gamma_{1} - (\zeta_{2}-\zeta_{1})^{3} |)^{b}}d\mu
		\\ \\&
		\leq C
		\int_{\mathbb{R}^{6}}\frac{| \zeta |^{s} | \zeta | m(\zeta,\gamma)} {(1+  |\gamma - \zeta^{3} |)^{-b'}}~
		\frac{(1+  | \zeta_{1} |)^{-s} f_{1} (\zeta_{1},\gamma_{1})} {(1+  |\gamma_{1} - \zeta_{1}^{3} |)^{b}}\\ \\&
		\frac{(1+  | \zeta-\zeta_{2}  |)^{-s} f_{2}(\zeta-\zeta_{2},\gamma-\gamma_{2})} {(1+  |\gamma -\gamma_{2} - (\zeta-\zeta_{2})^{3} |)^{b}}~
		\frac{(1+  | \zeta_{2}-\zeta_{1}  |)^{-s} f_{3}(\zeta_{2}-\zeta_{1},\gamma_{2}-\gamma_{1})} {(1+  |\gamma_{2}-\gamma_{1} - (\zeta_{2}-\zeta_{1})^{3} |)^{b}} d\mu \\ \\&
		\leq C
		\int_{\mathbb{R}^{6}}\frac{| \zeta |^{s+1}  m(\zeta,\gamma)} {(1+  |\gamma - \zeta^{3} |)^{-b'}}~~
		\frac{(1+  | \zeta_{1} |)^{-s} f_{1} (\zeta_{1},\gamma_{1})} {(1+  |\gamma_{1} - \zeta_{1}^{3} |)^{b}}\\ \\&
		\frac{(1+  | \zeta-\zeta_{2}  |)^{-s} f_{2}(\zeta-\zeta_{2},\gamma-\gamma_{2})} {(1+  |\gamma -\gamma_{2} - (\zeta-\zeta_{2})^{3} |)^{b}}
		\frac{(1+  | \zeta_{2}-\zeta_{1}  |)^{-s} f_{3}(\zeta_{2}-\zeta_{1},\gamma_{2}-\gamma_{1})} {(1+  |\gamma_{2}-\gamma_{1} - (\zeta_{2}-\zeta_{1})^{3} |)^{b}} d\mu.
	\end{align*}
	By
	$$
	| \zeta |^{s+1}(1+| \zeta_{1} |)^{-s} \leq | \zeta |^{s+1} | \zeta_{1} |^{-s}\leq C | \zeta |^{\frac{1}{2}} | \zeta_{1}  |^{\frac{1}{2}},
	$$
	we get
	\begin{align*}
		\| \partial_{x} (\phi_{1}   \phi_{2}  \phi_{3} )  \|_{X^{s, b'}} & \leq C
		\int_{\mathbb{R}^{6}}\frac{| \zeta |^{\frac{1}{2}}  m(\zeta,\gamma)} {(1+  |\gamma - \zeta^{3} |)^{-b'}}
		\frac{(1+  | \zeta_{1} |)^{\frac{1}{2}} f_{1} (\zeta_{1},\gamma_{1})} {(1+  |\gamma_{1} - \zeta_{1}^{3} |)^{b}}
		\\ \\&\frac{(1+  | \zeta-\zeta_{2}  |)^{-s} f_{2}(\zeta-\zeta_{2},\gamma-\gamma_{2})} {(1+  |\gamma -\gamma_{2} - (\zeta-\zeta_{2})^{3} |)^{b}}
		\frac{(1+  | \zeta_{2}-\zeta_{1}  |)^{-s} f_{3}(\zeta_{2}-\zeta_{1},\gamma_{2}-\gamma_{1})} {(1+  |\gamma_{2}-\gamma_{1} - (\zeta_{2}-\zeta_{1})^{3} |)^{b}} d\mu.
	\end{align*}
	
	We suppose that
	\begin{align*}
		\widehat{A^{\frac{1}{2}} M_{-b'}}(\zeta,\gamma)& = \frac{| \zeta |^{\frac{1}{2}}  m(\zeta,\gamma)} {(1+  |\gamma - \zeta^{3} |)^{-b'}}\\ \\
		\widehat{A^{\frac{1}{2}} F^{1}_{b}}(\zeta_{1},\gamma_{1})  &=  \frac{(1+  | \zeta_{1} |)^{\frac{1}{2}} f_{1} (\zeta_{1},\gamma_{1})} {(1+  |\gamma_{1} - \zeta_{1}^{3} |)^{b}}\\ \\
		\widehat{A^{-s} F^{2}_{b}}(\zeta-\zeta_{2},\gamma-\gamma_{2}) &=\frac{(1+  | \zeta-\zeta_{2}  |)^{-s} f_{2}(\zeta-\zeta_{2},\gamma-\gamma_{2})} {(1+  |\gamma -\gamma_{2} - (\zeta-\zeta_{2})^{3} |)^{b}}\\ \\
		\widehat{A^{-s} F^{3}_{b}}(\zeta_{2}-\zeta_{1},\gamma_{2}-\gamma_{1})&= \frac{(1+  | \zeta_{2}-\zeta_{1}  |)^{-s} f_{3}(\zeta_{2}-\zeta_{1},\gamma_{2}-\gamma_{1})} {(1+  |\gamma_{2}-\gamma_{1} - (\zeta_{2}-\zeta_{1})^{3} |)^{b}}.
	\end{align*}
	Then
	\begin{align*}
		&\int_{\mathbb{R}^{6}}\frac{( | \zeta |)^{\frac{1}{2}}  m(\zeta,\gamma)} {(1+  |\gamma - \zeta^{3} |)^{-b'}}~~~~
		\frac{(1+  | \zeta_{1} |)^{\frac{1}{2}} f_{1} (\zeta_{1},\gamma_{1})} {(1+  |\gamma_{1} - \zeta_{1}^{3} |)^{b}}~
		\frac{(1+  | \zeta-\zeta_{2}  |)^{-s} f_{2}(\zeta-\zeta_{2},\gamma-\gamma_{2})} {(1+  |\gamma -\gamma_{2} - (\zeta-\zeta_{2})^{3} |)^{b}}\\ \\&
		\frac{(1+  | \zeta_{2}-\zeta_{1}  |)^{-s} f_{3}(\zeta_{2}-\zeta_{1},\gamma_{2}-\gamma_{1})} {(1+  |\gamma_{2}-\gamma_{1} - (\zeta_{2}-\zeta_{1})^{3} |)^{b}} d\mu \\ \\&
		=	\int_{\mathbb{R}^{6}} \widehat{A^{\frac{1}{2}} M_{-b'}}(\zeta,\gamma)   \widehat{A^{\frac{1}{2}} F_{b}}(\zeta_{1},\gamma_{1})
		\widehat{A^{-s} G^{1}_{b}}(\zeta-\zeta_{2},\gamma-\gamma_{2})
		\widehat{A^{-s} G^{2}_{b}}(\zeta_{2}-\zeta_{1},\gamma_{2}-\gamma_{1})
		d\mu\\ \\&
		= \int_{\mathbb{R}^{2}}  \left(  \widehat{A^{\frac{1}{2}} M_{-b'}}(\zeta,\gamma) \right)\times\\ \\&  \left(  \int_{\mathbb{R}^{4}}  \widehat{A^{\frac{1}{2}} F_{b}}(\zeta_{1},\gamma_{1})
		\widehat{A^{-s} G^{1}_{b}}(\zeta-\zeta_{2},\gamma-\gamma_{2})
		\widehat{A^{-s} G^{2}_{b}}(\zeta_{2}-\zeta_{1},\gamma_{2}-\gamma_{1}) d\zeta_{1} d\gamma_{1} d\zeta_{2} d\gamma_{2} \right)
		d\zeta d\gamma
		\\ \\&=    \int_{\mathbb{R}^{2}}  \left(  \widehat{A^{\frac{1}{2}} M_{-b'}}(\zeta,\gamma) \right)  \left(  \left(  \widehat{A^{\frac{1}{2}} F^{1}_{b}}\star
		\widehat{A^{-s} F^{2}_{b}}\star
		\widehat{A^{-s} F^{3}_{b}} \right) (\zeta,\gamma) \right)
		d\zeta d\gamma
		\\ \\& =   \int_{\mathbb{R}^{2}}  \left(  \widehat{A^{\frac{1}{2}} M_{-b'}}(\zeta,\gamma) \right)  \left(  \left(  \widehat{A^{\frac{1}{2}} F^{1}_{b}\cdot A^{-s} F^{2}_{b}\cdot A^{-s} F^{3}_{b}}\right) (\zeta,\gamma) \right)
		d\zeta d\gamma
		\\ \\& =   \int_{\mathbb{R}^{2}}  A^{\frac{1}{2}}  M_{-b'} \left(  A^{\frac{1}{2}} F^{1}_{b}\cdot A^{-s} F^{2}_{b}\cdot A^{-s} F^{3}_{b} \right) dx dt.
	\end{align*}
	By using Cauchy-Schwarz's inequality  for the variables $x$  and $t$, we get
	\begin{align*}
		& \| \partial_{x} (\phi_{1}   \phi_{2}  \phi_{3} )  \|_{X^{s, b'}} \leq C  \| A^{\frac{1}{2}} M_{-b'} \|_{L_{x}^{4} L_{t}^{2} }   \| A^{\frac{1}{2}} F^{1}_{b} \|_{L_{x}^{4} L_{t}^{2} }  \| A^{-s}
		F^{2}_{b} \|_{L_{x}^{2} L_{t}^{\infty}}  \| A^{\frac{1}{2}} F^{3}_{b} \|_{L_{x}^{\infty} L_{t}^{\infty} }.
	\end{align*}
	Hence by Lemmas \ref{2.3} and \ref{2.4}
	\begin{align*}
		\| \partial_{x} (\phi_{1}   \phi_{2}  \phi_{3} )  \|_{X^{s, b'}} & \leq  C  \| m \|_{L_{\zeta}^{2} L_{\gamma}^{2} }   \| f \|_{L_{\zeta}^{2} L_{\gamma}^{2} }  \| f_{2} \|_{L_{\zeta}^{2} L_{\gamma}^{2} }   \| f_{3} \|_{L_{\zeta}^{2} L_{\gamma}^{2} } \\ \\ &
		\leq C\| \phi_{1}\|_{X^{s, b}}  \| \phi_{2} \|_{X^{s, b}}  \| \phi_{3} \|_{X^{s, b}}.
	\end{align*}
	\textbf{Secondly for  the case $  |\zeta| \leq 1 $:} We have
	\begin{align*}
		(1+  | \zeta |)^{s} | \zeta | (1+  | \zeta_{1} |)^{-s}  \leq  C   (1+  | \zeta |)^{^{\frac{1}{2}}} (1+  | \zeta_{1} |)^{^{\frac{1}{2}}},
	\end{align*}
	then
	\begin{align*}	
		\| \partial_{x} (\phi_{1}   \phi_{2}  \phi_{3} )  \|_{X^{s, b'}}& \leq \int_{\mathbb{R}^{6}}\frac{(1+  | \zeta |)^{s} | \zeta | m(\zeta,\gamma)} {(1+  |\gamma - \zeta^{3} |)^{-b'}}~
		\frac{(1+  | \zeta_{1} |)^{-s} f_{1} (\zeta_{1},\gamma_{1})} {(1+  |\gamma_{1} - \zeta_{1}^{3} |)^{b}}\\ \\&\times
		\frac{(1+  | \zeta-\zeta_{2}  |)^{-s} f_{2}(\zeta-\zeta_{2},\gamma-\gamma_{2})} {(1+  |\gamma -\gamma_{2} - (\zeta-\zeta_{2})^{3} |)^{b}}
		\frac{(1+  | \zeta_{2}-\zeta_{1}  |)^{-s} f_{3}(\zeta_{2}-\zeta_{1},\gamma_{2}-\gamma_{1})} {(1+  |\gamma_{2}-\gamma_{1} - (\zeta_{2}-\zeta_{1})^{3} |)^{b}}d\mu
		\\ \\&
		\leq C
		\int_{\mathbb{R}^{6}}\frac{  (1+ | \zeta  |)^{\frac{1}{2}} m(\zeta,\gamma)} {(1+  |\gamma - \zeta^{3} |)^{-b'}}~
		\frac{(1+  | \zeta_{1} |)^{\frac{1}{2}} f_{1} (\zeta_{1},\gamma_{1})} {(1+  |\gamma_{1} - \zeta_{1}^{3} |)^{b}}\\ \\& \quad\times
		\frac{(1+  | \zeta-\zeta_{2}  |)^{-s} f_{2}(\zeta-\zeta_{2},\gamma-\gamma_{2})} {(1+  |\gamma -\gamma_{2} - (\zeta-\zeta_{2})^{3} |)^{b}}~
		\frac{(1+  | \zeta_{2}-\zeta_{1}  |)^{-s} f_{3}(\zeta_{2}-\zeta_{1},\gamma_{2}-\gamma_{1})} {(1+  |\gamma_{2}-\gamma_{1} - (\zeta_{2}-\zeta_{1})^{3} |)^{b}} d\mu.
	\end{align*}
	Then, by the inner product, we have
	\begin{align*}
		\| \partial_{x} (\phi_{1}   \phi_{2}  \phi_{3} )  \|_{X^{s, b'}} &
		\leq C\langle A^{\frac{1}{2}} M_{-b'};  A^{\frac{1}{2}} F^{1}_{b}\cdot A^{-s} F^{2}_{b}\cdot A^{-s} F^{3}_{b}  \rangle   \\ &
		\leq  C  \| A^{\frac{1}{2}} M_{-b'} \|_{L_{x}^{4} L_{t}^{2} }   \| A^{\frac{1}{2}} F^{1}_{b} \|_{L_{x}^{4} L_{t}^{2} }  \| A^{-s}
		F^{2}_{b} \|_{L_{x}^{2} L_{t}^{\infty}}  \| A^{-s} F^{3}_{b} \|_{L_{x}^{\infty} L_{t}^{\infty} }.
	\end{align*}
	Hence by Lemmas \ref{2.3} and \ref{2.4}
	\begin{align*}
		\| \partial_{x} (\phi_{1}   \phi_{2}  \phi_{3} )  \|_{X^{s, b'}} & \leq  C  \| m \|_{L_{\zeta}^{2} L_{\gamma}^{2} }   \| f_1 \|_{L_{\zeta}^{2} L_{\gamma}^{2} }  \| f_{2} \|_{L_{\zeta}^{2} L_{\gamma}^{2} }   \| f_{3} \|_{L_{\zeta}^{2} L_{\gamma}^{2} } \\ \\ &
		\leq C\| \phi_{1}\|_{X^{s, b}}  \| \phi_{2} \|_{X^{s, b}}  \| \phi_{3} \|_{X^{s, b}}.
	\end{align*}
	The inequalities in the other five regions are proved in a similar manner.\\
	The case $  k\geq  2 $ is handled virtually identically. The only difference is that  we need to split the Fourier space in  $((2k+1)+1)!$.\\
	We want to prove that
	\begin{align*}
		\bigg\| \partial_{x}\prod _{i=1}^{2k+1} \phi_{i} \bigg\|_{X^{s, b'}}\leq
		C  \prod _{i=1}^{2k+1}\|  \phi_{i}  \|_{X^{s, b}},
	\end{align*}
	We have 
	\begin{align*}
		\bigg\| \partial_{x}\prod _{i=1}^{2k+1} \phi_{i} \bigg\|_{X^{s, b'}} & =  \bigg\|  (1+  | \zeta |)^{s} (1+  |\gamma - \zeta^{3} |)^{b'} \widehat{ \partial_{x}\prod _{i=1}^{2k+1} \phi_{i}  }(\zeta, \gamma  ) \bigg\|_{L^{2}_{\zeta} L^{2}_{\gamma} },
		\\ &=  \bigg\|  (1+  | \zeta |)^{s} (1+  |\gamma - \zeta^{3} |)^{b'} |\zeta|      \widehat{ \prod _{i=1}^{2k+1} \phi_{i}  }(\zeta, \gamma  ) \bigg\|_{L^{2}_{\zeta} L^{2}_{\gamma} }.
	\end{align*}
	By the same way,  by the inner product, we have
	\begin{align*}
		\bigg\| \partial_{x}\prod _{i=1}^{2k+1} \phi_{i} \bigg\|_{X^{s, b'}}	& \leq C  \langle \widehat{A^{\frac{1}{2}} M_{-b'}}; \widehat{A^{\frac{1}{2}} F^{1}_{b}}\star  \widehat{A^{-s} F^{2}_{b}}  \star  \prod _{i=3}^{2k+1} \widehat{A^{-s} F^{i}_{b}}\rangle \\ &
		\leq C  \langle \widehat{A^{\frac{1}{2}} M_{-b'}}; \widehat{A^{\frac{1}{2}} F^{1}_{b}\cdot A^{-s} F^{2}_{b}  \prod _{i=3}^{2k+1} A^{-s} F^{i}_{b}}\rangle \\ &
		\leq  C  \| A^{\frac{1}{2}} M_{-b'} \|_{L_{x}^{4} L_{t}^{2} }   \| A^{\frac{1}{2}} F^{1}_{b} \|_{L_{x}^{4} L_{t}^{2} } \| A^{-s} F^{2}_{b}  \|_{L_{x}^{2} L_{t}^{\infty}}  \bigg\| \prod _{i=3}^{2k+1}A^{-s}
		F^{i}_{b} \bigg\|_{L_{x}^{\infty} L_{t}^{\infty}}\\ &
		\leq  C  \prod _{i=1}^{2k+1}\|  \phi_{i}  \|_{X^{s, b}}.
	\end{align*}
	\textbf{\underline{Inequality (\ref{mil2}):}} We have	
	
	\begin{align*}
		\| \partial_{x} (\varphi_{1}  \varphi_{2} \varphi_{3} ) \|_{X_{\alpha}^{s, b'}}	& \leq \int_{\mathbb{R}^{6}}\frac{(1+  | \zeta |)^{s} | \zeta | m(\zeta,\gamma)} {(1+  |\gamma - \alpha\zeta^{3} |)^{-b'}}
		\frac{(1+  | \zeta_{1} |)^{-s} g_{1} (\zeta_{1},\gamma_{1})} {(1+  |\gamma_{1} - \alpha\zeta_{1}^{3} |)^{b}} \\ \\&
		\frac{(1+  | \zeta-\zeta_{2}  |)^{-s} g_{2}(\zeta-\zeta_{2},\gamma-\gamma_{2})} {(1+  |\gamma -\gamma_{2} - \alpha(\zeta-\zeta_{2})^{3} |)^{b}}
		\frac{(1+  | \zeta_{2}-\zeta_{1}  |)^{-s} g_{3}(\zeta_{2}-\zeta_{1},\gamma_{2}-\gamma_{1})} {(1+  |\gamma_{2}-\gamma_{1} - \alpha(\zeta_{2}-\zeta_{1})^{3} |)^{b}}  d \mu.
	\end{align*}
	\textbf{We begin by the region 1:} We assume that $  | \zeta| \leq 1  $	or    $  | \zeta| \geq 1  $. \\
	\textbf{Firstly, case $  | \zeta| \geq 1   $:}
	\begin{align*}
		\| \partial_{x} (\varphi_{1}  \varphi_{2} \varphi_{3} ) \|_{X_{\alpha}^{s, b'}} & \leq C
		\int_{\mathbb{R}^{6}}\frac{| \zeta |^{\frac{1}{2}}  m(\zeta,\gamma)} {(1+  |\gamma - \alpha\zeta^{3} |)^{-b'}}
		\frac{(1+  | \zeta_{1} |)^{\frac{1}{2}} g_{1} (\zeta_{1},\gamma_{1})} {(1+  |\gamma_{1} - \alpha\zeta_{1}^{3} |)^{b}}\\& \quad\times
		\frac{(1+  | \zeta-\zeta_{2}  |)^{-s} g_{2}(\zeta-\zeta_{2},\gamma-\gamma_{2})} {(1+  |\gamma -\gamma_{2} - \alpha(\zeta-\zeta_{2})^{3} |)^{b}}
		\frac{(1+  | \zeta_{2}-\zeta_{1}  |)^{-s} g_{3}(\zeta_{2}-\zeta_{1},\gamma_{2}-\gamma_{1})} {(1+  |\gamma_{2}-\gamma_{1} - \alpha(\zeta_{2}-\zeta_{1})^{3} |)^{b}} d\mu.
	\end{align*}
	
	We suppose that
	\begin{align*}
		\widehat{A^{\frac{1}{2}} M^{\alpha}_{-b'}}(\zeta,\gamma)& = \frac{| \zeta |^{\frac{1}{2}}  m(\zeta,\gamma)} {(1+  |\gamma - \alpha\zeta^{3} |)^{-b'}}\\ \\
		\widehat{A^{\frac{1}{2}} G^{1}_{\alpha,b}}(\zeta_{1},\gamma_{1})  &=  \frac{(1+  | \zeta_{1} |)^{\frac{1}{2}} g_{1} (\zeta_{1},\gamma_{1})} {(1+  |\gamma_{1} - \alpha\zeta_{1}^{3} |)^{b}}\\ \\
		\widehat{A^{-s} G^{2}_{\alpha,b}}(\zeta-\zeta_{2},\gamma-\gamma_{2}) &=\frac{(1+  | \zeta-\zeta_{2}  |)^{-s} g_{2}(\zeta-\zeta_{2},\gamma-\gamma_{2})} {(1+  |\gamma -\gamma_{2} - \alpha(\zeta-\zeta_{2})^{3} |)^{b}}\\ \\
		\widehat{A^{-s} G^{3}_{\alpha,b}}(\zeta_{2}-\zeta_{1},\gamma_{2}-\gamma_{1})&= \frac{(1+  | \zeta_{2}-\zeta_{1}  |)^{-s} g_{3}(\zeta_{2}-\zeta_{1},\gamma_{2}-\gamma_{1})} {(1+  |\gamma_{2}-\gamma_{1} - \alpha(\zeta_{2}-\zeta_{1})^{3} |)^{b}}.
	\end{align*}
	Then

	\begin{align*}
		&\int_{\mathbb{R}^{6}}\frac{| \zeta |^{\frac{1}{2}}  m(\zeta,\gamma)} {(1+  |\gamma - \alpha\zeta^{3} |)^{-b'}}
		\frac{(1+  | \zeta_{1} |)^{\frac{1}{2}} g_{1} (\zeta_{1},\gamma_{1})} {(1+  |\gamma_{1} - \alpha\zeta_{1}^{3} |)^{b}}
		\frac{(1+  | \zeta-\zeta_{2}  |)^{-s} g_{2}(\zeta-\zeta_{2},\gamma-\gamma_{2})} {(1+  |\gamma -\gamma_{2} - \alpha(\zeta-\zeta_{2})^{3} |)^{b}}\\ \\& \quad\times
		\frac{(1+  | \zeta_{2}-\zeta_{1}  |)^{-s} g_{3}(\zeta_{2}-\zeta_{1},\gamma_{2}-\gamma_{1})} {(1+  |\gamma_{2}-\gamma_{1} - \alpha(\zeta_{2}-\zeta_{1})^{3} |)^{b}} d\mu \\ \\&
		=	\int_{\mathbb{R}^{6}} \widehat{A^{\frac{1}{2}} M^{\alpha}_{-b'}}(\zeta,\gamma)   \widehat{A^{\frac{1}{2}} G^{1}_{\alpha,b}}(\zeta_{1},\gamma_{1})
		\widehat{A^{-s} G^{2}_{\alpha,b}}(\zeta-\zeta_{2},\gamma-\gamma_{2})
		\widehat{A^{-s} G^{3}_{\alpha,b}}(\zeta_{2}-\zeta_{1},\gamma_{2}-\gamma_{1})
		d\mu
		\\ \\& =   \int_{\mathbb{R}^{2}}  A^{\frac{1}{2}}  M^{\alpha}_{-b'} \left(  A^{\frac{1}{2}} G^{1}_{\alpha,b}\cdot A^{-s} G^{2}_{\alpha,b}\cdot A^{-s} G^{3}_{\alpha,b} \right) dx dt.
	\end{align*}
	By using Cauchy-Schwarz's inequality  for the variables $x$  and $t$, we get
	\begin{align*}
		& \| \partial_{x} (\varphi_{1}  \varphi_{2} \varphi_{3} ) \|_{X_{\alpha}^{s, b'}}  \leq C  \| A^{\frac{1}{2}} M^{\alpha}_{-b'} \|_{L_{x}^{4} L_{t}^{2} }   \| A^{\frac{1}{2}} G^{1}_{\alpha,b} \|_{L_{x}^{4} L_{t}^{2} }  \| A^{-s}
		G^{2}_{\alpha,b} \|_{L_{x}^{2} L_{t}^{\infty}}  \| A^{\frac{1}{2}} G^{3}_{\alpha,b} \|_{L_{x}^{\infty} L_{t}^{\infty} }.
	\end{align*}
	Hence by Lemmas \ref{2.3} and \ref{2.5}
	\begin{align*}
		\| \partial_{x} (\varphi_{1}  \varphi_{2} \varphi_{3} ) \|_{X_{\alpha}^{s, b'}} & \leq  C  \| m \|_{L_{\zeta}^{2} L_{\gamma}^{2} }   \| g_1 \|_{L_{\zeta}^{2} L_{\gamma}^{2} }  \| g_{2} \|_{L_{\zeta}^{2} L_{\gamma}^{2} }   \| g_{3} \|_{L_{\zeta}^{2} L_{\gamma}^{2} } \\ \\ &
		\leq C\| \varphi_{1}\|_{X_{\alpha}^{s, b}}  \| \varphi_{2} \|_{X_{\alpha}^{s, b}}  \| \varphi_{3} \|_{X_{\alpha}^{s, b}}.
	\end{align*}
	\textbf{Secondly for  the case $  |\zeta| \leq 1 $:} Similar, here we use 
	\begin{align*}
		(1+  | \zeta |)^{s} | \zeta | (1+  | \zeta_{1} |)^{-s}  \leq  C   (1+  | \zeta |)^{^{\frac{1}{2}}} (1+  | \zeta_{1} |)^{^{\frac{1}{2}}}.
	\end{align*}
	\textbf{Case $  k\geq  2 $:} we have 
	\begin{align*}
		\bigg\| \partial_{x}\prod _{i=1}^{2k+1} \varphi_{i} \bigg\|_{X_{\alpha}^{s, b'}} & =  \bigg\|  (1+  | \zeta |)^{s} (1+  |\gamma - \alpha\zeta^{3} |)^{b'} \widehat{ \partial_{x}\prod _{i=1}^{2k+1} \varphi_{i}  }(\zeta, \gamma  ) \bigg\|_{L^{2}_{\zeta} L^{2}_{\gamma} },
		\\ &=  \bigg\|  (1+  | \zeta |)^{s} (1+  |\gamma - \alpha\zeta^{3} |)^{b'} |\zeta|      \widehat{ \prod _{i=1}^{2k+1} \varphi_{i}  }(\zeta, \gamma  ) \bigg\|_{L^{2}_{\zeta} L^{2}_{\gamma} }.
	\end{align*}
	By the same way,  by the inner product, we have
	\begin{align*}
		\bigg\| \partial_{x}\prod _{i=1}^{2k+1} \varphi_{i} \bigg\|_{X_{\alpha}^{s, b'}}	& \leq C  \langle \widehat{A^{\frac{1}{2}} M^{\alpha}_{-b'}}; \widehat{A^{\frac{1}{2}} G^{1}_{\alpha,b}}\star  \widehat{A^{-s} G^{2}_{\alpha,b}}  \star  \prod _{i=3}^{2k+1} \widehat{A^{-s} G^{i}_{\alpha,b}}\rangle \\ &
		\leq C  \langle \widehat{A^{\frac{1}{2}} M^{\alpha}_{-b'}}; \widehat{A^{\frac{1}{2}} G^{1}_{\alpha,b}\cdot A^{-s} G^{2}_{\alpha,b}  \prod _{i=3}^{2k+1} A^{-s} G^{i}_{\alpha,b}}\rangle \\ &
		\leq  C  \| A^{\frac{1}{2}} M^{\alpha}_{-b'} \|_{L_{x}^{4} L_{t}^{2} }   \| A^{\frac{1}{2}} G^{1}_{\alpha,b} \|_{L_{x}^{4} L_{t}^{2} } \| A^{-s} G^{2}_{\alpha,b}  \|_{L_{x}^{2} L_{t}^{\infty}}  \bigg\| \prod _{i=3}^{2k+1}A^{-s}
		G^{i}_{\alpha,b} \bigg\|_{L_{x}^{\infty} L_{t}^{\infty}}\\ &
		\leq  C  \prod _{i=1}^{2k+1}\|  \varphi_{i}  \|_{X_\alpha^{s, b}}.
	\end{align*}
\end{proof}
We choose a function $\varpi$ such that $\varpi(t)=0$ for $t<0$ and $|t|>2$ ,$ \varpi(t)=1$ for $t \in [0, 1]$ , with $\varpi\in C_{0}^{\infty} $. Note that such a $\varpi$  belongs to $H_t^b=H^b([0,T],\mathbb{R})$ for any $b>\frac{1}{2}$, where 
\[\|\varpi\|^2_{H_t^b}=\|\varpi\|^2_{L^2}+\int_{\mathbb{R}^{2}}\dfrac{|\varpi(\eta_1)-\varpi(\eta_2)|^{2}}{|\eta_1-\eta_2|^{1+2b}}d\eta_1d\eta_2.\]
We show the following lemma.

\begin{lemma}[Stochastic convolution]\label{lem2}
	Let $s, b \in \mathbb{R}$, with $b>\frac{1}{2}$, and assume that $\Xi \in L_2^{0, s}$, then $\phi_{l}(t), \varphi_{l}(t)$ defined by (\ref{sto}) satisfies
	$$
	\varpi \phi_{l} \in L^2\left(\Omega, X^{b, s}\right), \quad 	\varpi \varphi_{l} \in L^2\left(\Omega, X_{\alpha}^{b, s}\right)
	$$
	and
	$$
	\mathbb{E}\left(\|\varpi \phi_{l}\|_{X^{b, s}}^2\right) \lesssim_{ b, \varpi}\|\Xi\|_{L_2^{0, s}}^2, \quad \mathbb{E}\left(\|\varpi \varphi_{l}\|_{X_{\alpha}^{b, s}}^2\right) \lesssim_{ b, \varpi}\|\Xi\|_{L_2^{0, s}}^2.
	$$
	
\end{lemma}
\begin{proof}
	Let us introduce the function
	\begin{equation}\label{eq11}
		f(\cdot,t)=\varpi(t) \int_0^t \mathfrak{U}_{\alpha}(-\gamma) \Xi d W(\gamma), \quad t \in \mathbb{R}^{+} .
	\end{equation}

	This implies that $\mathfrak{U}_{\alpha}(t) f(\cdot,t)=\varpi(t) \Psi(t)$. Thus, we have
	$$
	\begin{aligned}
		\mathbb{E}\left(\left\|\varpi \Psi\right\|_{X_{s, b}}^2\right) & =\mathbb{E}\left(\int_{\mathbb{R}^{2}} (1+|\zeta|)^{2 s}(1+|\gamma|)^{2 b}\left|\hat{f}(\zeta,t)\right|^2 d \gamma d \zeta\right) \\
		& =\int_{\mathbb{R}}(1+|\zeta|)^{2 s} \mathbb{E}\left(\left\|\hat{f}(\zeta,\cdot)\right\|_{H_t^b}^2\right) d \zeta,
	\end{aligned}
	$$
	
	According to the expansion $W(t)=\sum_{i=0}^{\infty}\beta_i(t)e_i$ of the cylindrical Wiener process and (\ref{sto})$_{2}$, we have
	$$
	\mathbb{E}\left(\left\|\hat{f}(\zeta,\cdot)\right\|_{H_t^b}^2\right)=\mathfrak{S}_1+\mathfrak{S}_2
	$$
	where,
	$$
	\mathfrak{S}_1=\sum_{i=0}^{\infty}\left|\hat{\Xi e_i} \right|^2\left[\mathbb{E}\left(\left\|\varpi(t) \int_0^t e^{i \gamma \alpha \zeta^{3}} d \beta_i(\gamma)\right\|_{L^2(\mathbb{R})}^2\right)\right],
	$$
	
	$$
	\mathfrak{S}_2=\sum_{i=0}^{\infty}\left|\hat{\Xi e_i}\right|^2\left[\mathbb{E}\left(\int_{\mathbb{R}^{2}}  \frac{\left|\begin{array}{c}
			\varpi\left(\eta_1\right) \int_0^{\eta_1} e^{i \gamma \alpha \zeta^{3}} d \beta_i(\gamma) 
			-\varpi\left(\eta_2\right) \int_0^{\eta_2} e^{i \gamma \alpha \zeta^{3}} d \beta_i(\gamma)
		\end{array}\right|}{\left|\eta_1-\eta_2\right|^{1+2 b}} d \eta_1 d \eta_2\right)\right] .
	$$
	
	From the Itô isometry formula, we have
	$$
	\begin{aligned}
		\mathfrak{S}_1 & =\sum_{i=0}^{\infty}\left|\hat{\Xi e_i}\right|^2 \int_0^2|\varpi(t)|^2 \mathbb{E}\left(\left|\int_0^t e^{i \gamma \alpha \zeta^{3}} d \beta_i(\gamma)\right|^2\right) d t \\
		& =\left\||t|^{\frac{1}{2}} \varpi\right\|_{L_t^2}^2 \sum_{i=0}^{\infty}\left|\hat{\Xi e_i}\right|^2 .
	\end{aligned}
	$$
	
	To estimate $\mathfrak{S}_2$, we get
	
	$$
	\begin{aligned}
		& \mathfrak{S}_2=\sum_{i=0}^{\infty}\left|\hat{\Xi e_i}\right|^2\left[\mathbb{E}\left(\int_{\mathbb{R}^{2}} \frac{\left|\begin{array}{c}
				\varpi\left(\eta_1\right) \int_0^{\eta_1} e^{i \gamma \alpha \zeta^{3}} d \beta_i(\gamma) 
				-\varpi\left(\eta_2\right) \int_0^{\eta_2} e^{i \gamma \alpha \zeta^{3}} d \beta_i(\gamma)
			\end{array}\right|}{\left|\eta_1-\eta_2\right|^{1+2 b}} d \eta_1 d \eta_2\right)\right] \\
		& =2 \sum_{i=0}^{\infty}\left|\hat{\Xi e_i}\right|^2 \int_{\eta_2>0} \int_{\eta_1<\eta_2} \frac{\mathbb{E}\left(\left|\begin{array}{c}
				\varpi\left(\eta_1\right) \int_0^{\eta_1} e^{i \gamma \alpha \zeta^{3}} d \beta_i(\gamma) 
				-\varpi\left(\eta_2\right) \int_0^{\eta_2} e^{i \gamma \alpha \zeta^{3}} d \beta_i(\gamma)
			\end{array}\right|^2\right.}{\left|\eta_1-\eta_2\right|^{1+2 b}} d \eta_1 d \eta_2 \\
		& \leq \sum_{i=0}^{\infty}\left|\hat{\Xi e_i}\right|^2\left[2 \int_{\eta_2>0} \int_{\eta_1<0} \frac{\left|\varpi\left(\eta_2\right)\right|^2 \mathbb{E}\left(\left|\int_0^{\eta_2} e^{i \gamma \alpha \zeta^{3}} d \beta_i(\gamma)\right|^2\right)}{\left|\eta_1-\eta_2\right|^{1+2 b}} d \eta_1 d \eta_2\right. \\
		& \left.+2 \int_{\eta_2>0} \int_{0<\eta_1<\eta_2} \frac{\mathbb{E}\left(\left\lvert\, \begin{array}{c}
				\varpi\left(\eta_1\right) \int_0^{\eta_1} e^{i \gamma \alpha \zeta^{3}} d \beta_i(\gamma) 
				-\varpi\left(\eta_2\right) \int_0^{\eta_1} e^{i \gamma \alpha \zeta^{3}} d \beta_i(\gamma) 
				+\varpi\left(\eta_2\right) \int_{\eta_1}^{\eta_2} e^{i \gamma \alpha \zeta^{3}} d \beta_i(\gamma)
			\end{array}\right.\right)^2}{\left|\eta_1-\eta_2\right|^{1+2 b}} d \eta_1 d \eta_2\right]\\ 
		& \leq \sum_{i=0}^{\infty}\left|\hat{\Xi e_i}\right|^2\left[2 \int_{\eta_2>0} \int_{\eta_1<0} \frac{\left|\varpi\left(\eta_2\right)\right|^2 \mathbb{E}\left(\left|\int_0^{\eta_2} e^{i \gamma \alpha \zeta^{3}} d \beta_i(\gamma)\right|^2\right)}{\left|\eta_1-\eta_2\right|^{1+2 b}} d \eta_1 d \eta_2\right. \\
		& +4 \int_{\eta_2>0} \int_{0<\eta_1<\eta_2} \frac{\left|\varpi\left(\eta_1\right)-\varpi\left(\eta_2\right)\right|^2 \mathbb{E}\left(\left|\int_0^{\eta_1} e^{i \gamma \alpha \zeta^{3}} d \beta_i(\gamma)\right|^2\right)}{\left|\eta_1-\eta_2\right|^{1+2 b}} d \eta_1 d \eta_2 \\
		& \left.+4 \int_{\eta_2>0} \int_{0<\eta_1<\eta_2} \frac{\left|\varpi\left(\eta_2\right)\right|^2 \mathbb{E}\left(\left|\int_{\eta_1}^{\eta_2} e^{i \gamma \alpha \zeta^{3}} d \beta_i(\gamma)\right|^2\right)}{\left|\eta_1-\eta_2\right|^{1+2 b}} d \eta_1 d \eta_2\right] \\
		& =\sum_{i=0}^{\infty}\left|\hat{\Xi e_i}\right|^2\left[I_1+I_2+I_3\right] . \\
		&
	\end{aligned}
	$$
	We now bound $\mathfrak{I}_1, \mathfrak{I}_2$, and $\mathfrak{I}_3$ separately,
	
	$$
	\mathfrak{I}_1 \leq 2 \int_0^2 \eta_1\left|\varpi\left(\eta_2\right)\right|^2 \int_{\eta_1<0} \frac{1}{\left|\eta_1-\eta_2\right|^{1+2 b}} d \eta_1 d \eta_2 \leq \mathfrak{M}_b\left\||t|^{\frac{1}{2}-b} \varpi\right\|_{L_t^2}^2 .
	$$
	
	Using Eq. (\ref{eq11}) and the assumption that $2 b \in(0,1)$, we have
	$$
	\begin{aligned}
		\mathfrak{I}_2 & \leq 4 \int_0^{\infty} \int_0^{\eta_2} \frac{\eta_1\left|\varpi\left(\eta_1\right)-\varpi\left(\eta_2\right)\right|^2}{\left\|\eta_1-\eta_2\right\|^{1+2 b}} d \eta_1 d \eta_2 \\
		& \leq 4 \int_0^2 \int_0^{\eta_2} \frac{\eta_1\left|\varpi\left(\eta_1\right)-\varpi\left(\eta_2\right)\right|^2}{\left\|\eta_1-\eta_2\right\|^{1+2 b}} d \eta_1 d \eta_2 \\
		& +4 \int_{2}^{\infty} \int_0^2 \frac{\eta_1\left|\varpi\left(\eta_1\right)\right|^2}{\left\|\eta_1-\eta_2\right\|^{1+2 b}} d \eta_1 d \eta_2 \\
		& \leq 8\|\varpi\|_{H_t^b}^2+4\left\||t|^{\frac{1}{2}} \varpi\right\|_{L_t^{\infty}}^2 \int_0^{\infty} \int_0^2 \frac{1}{\left|\eta_1-\eta_2\right|^{1+2 b}} d \eta_1 d \eta_2 \\
		& \leq 8\|\varpi\|_{H_t^b}^2+\mathfrak{M}_b\left\||t|^{\frac{1}{2}} \varpi\right\|_{L_t^{\infty}}^2 .
	\end{aligned}
	$$
	
	Similarly,
	$$
	\mathfrak{I}_3 \leq 4 \int_0^2 \int_0^{\eta_2} \frac{\left|\varpi\left(\eta_2\right)\right|^2}{\left|\eta_1-\eta_2\right|^{2 b}} d \eta_1 d \eta_2 \leq \mathfrak{M}_b\left\||t|^{\frac{1}{2}-b} \varpi\right\|_{L_t^2}^2 .
	$$
	
	So, we have
	$$
	\mathbb{E}\left(\left\|\hat{f}(\zeta,\cdot)\right\|_{H_t^b}^2\right) \leq \mathfrak{K}(b, \varpi) \sum_{i=0}^{\infty}\left|\hat{\Xi e_i}\right|^2
	$$
	where $\mathfrak{K}(b, \varpi)=\mathfrak{M}_b\left(\|\varpi\|_{H_t^b}+\left\||t|^{\frac{1}{2}} \varpi\right\|_{L_t^2}+\left\||t|^{\frac{1}{2}} \varpi\right\|_{L_t^{\infty}}\right)$.
\end{proof}
\section{Local well-posedness} 
Using the stochastic estimates from the previous section and the Banach fixed-point theorem, we now establish a local well-posedness result for system (\ref{p1}). This section is therefore dedicated to the proof of Theorem \ref{thm1}. Let
\[\varPi_1(t)=\mathfrak{U}(t)u_0,\qquad \varPi_2(t)=\mathfrak{U}_\alpha(t)\varphi_0,\]
\[\Upsilon_1(t)=\int_{0}^{t}\mathfrak{U}(t-\gamma)\Xi dW(\gamma), \qquad \Upsilon_2(t)=\int_{0}^{t}\mathfrak{U}_{\alpha}(t-\gamma)\Xi dW(\gamma),\]
and
\[\Psi_1(t)=\int_{0}^{t} \mathfrak{U}(t-\gamma)\partial_{x}(N_\phi(\phi,\varphi))(\gamma) d\gamma,\qquad \Psi_2(t)=\int_{0}^{t} \mathfrak{U}_{\alpha}(t-\gamma)\partial_{x}(N_\varphi(\phi,\varphi))(\gamma) d\gamma.\]
Then, we may rewrite (\ref{int2}) in terms of 
$$\Psi_1=\phi- \Upsilon_1-\varPi_1,$$
$$\Psi_2=\varphi- \Upsilon_2-\varPi_2,$$
as

\begin{equation}\label{sys1s}
	\left\{
	\begin{array}{ll}
		\Psi_1(t)=\dfrac{1}{2}\int_{0}^{t} \mathfrak{U}(t-\gamma)\partial_{x}(N_\phi(\phi,\varphi))(\gamma) d\gamma\\ \\
		\Psi_2(t)=\int_{0}^{t} \mathfrak{U}_{\alpha}(t-\gamma)\partial_{x}(N_\varphi(\phi,\varphi))(\gamma) d\gamma,
	\end{array}
	\right.
\end{equation}
with
\begin{equation}\label{sys1s1}
	\left\{\begin{array}{l}
		N_\phi(\phi, \varphi)=A \phi^{2 k+1}+B \phi^k \varphi^{k+1}+\frac{k+2}{k} C \phi^{k+1} \varphi^k+D \phi^{k-1} \varphi^{k+2} \\ \\
		N_\varphi(\phi, \varphi)=A \varphi^{2 k+1}+B \varphi^k \phi^{k+1}+\frac{k+2}{k} D \varphi^{k+1} \phi^k+C \varphi^{k-1} \phi^{k+2},
	\end{array}\right.
\end{equation}
then (\ref{sys1s})-(\ref{sys1s1})  is equivalent to
\begin{equation}\label{sys2s}
	\left\{
	\begin{array}{ll}
		\Psi_1(t)&=\dfrac{1}{2}\int_{0}^{t} \mathfrak{U}(t-\gamma)\partial_{x}\bigg(A (\Psi_1+\Upsilon_1+\varPi_1)^{2 k+1}+B (\Psi_1+\Upsilon_1+\varPi_1)^k (\Psi_2+\Upsilon_2+\varPi_2)^{k+1}
		\\ \\
		&+\frac{k+2}{k} C (\Psi_1+\Upsilon_1+\varPi_1)^{k+1} (\Psi_2+\Upsilon_2+\varPi_2)^k+D (\Psi_1+\Upsilon_1+\varPi_1)^{k-1} (\Psi_2+\Upsilon_2+\varPi_2)^{k+2}\bigg)(\gamma) d\gamma 	\\ \\
		\Psi_2(t)&=\int_{0}^{t} \mathfrak{U}_{\alpha}(t-\gamma)\partial_{x}\bigg(A (\Psi_2+\Upsilon_2+\varPi_2)^{2 k+1}+B (\Psi_2+\Upsilon_2+\varPi_2)^k (\Psi_1+\Upsilon_1+\varPi_1)^{k+1}
		\\ \\
		&+\frac{k+2}{k} D (\Psi_2+\Upsilon_2+\varPi_2)^{k+1} (\Psi_1+\Upsilon_1+\varPi_1)^k+C (\Psi_2+\Upsilon_2+\varPi_2)^{k-1} (\Psi_1+\Upsilon_1+\varPi_1)^{k+2}\bigg)(\gamma) d\gamma ,
	\end{array}
	\right.
\end{equation}
Next, we define the ball $ \mathcal{B}_{\mathcal{R},T} $, by
\begin{equation}
	\mathcal{B}_{\mathcal{R},T}=\{(\Psi_1,\Psi_2):\|(\Psi_1,\Psi_2)\|_{\mathcal{X}^{b, s}_{\alpha,T}}\leq \mathcal{R}, \quad \mathcal{R}>0\}.\nonumber
\end{equation}
The goal of this section is thus to prove that the mapping $(\Psi_1(t),\Psi_2(t))$ is a contraction on $\mathcal{B}_{\mathcal{R},T}$.
According to Lemmas \ref{lem1}, \ref{tri}, and \ref{lem2}, we obtain

\begin{eqnarray*}
	\|\varGamma_1(\Psi_1(t))\|_{X^{b, s}_{T}}&\lesssim& T^{1-b+b^{\prime}}\left(\mathcal{R}^{2k+1}+\|(\Upsilon_1,\Upsilon_2)\|^{2k+1}_{\mathcal{X}^{b, s}_{\alpha, T}}+\|(\phi_0, \varphi_0)\|^{2}_{\mathcal{H}^{ s }} \right),\\ \\ \|\varGamma_2(\Psi_2(t))\|_{X^{b, s}_{\alpha, T}}&\lesssim& T^{1-b+b^{\prime}}\left(\mathcal{R}^{2k+1}+\|(\Upsilon_1,\Upsilon_2)\|^{2k+1}_{\mathcal{X}^{b, s}_{\alpha, T}}+\|(\phi_0, \varphi_0)\|^{2}_{\mathcal{H}^{ s }} \right),
\end{eqnarray*}
so
\begin{eqnarray*}
	\|(\varGamma_1(\Psi_1(t)),\varGamma_2(\Psi_2(t)))\|_{\mathcal{X}^{b, s}_{\alpha, T}}&\lesssim& T^{1-b+b^{\prime}}\left(\mathcal{R}^{2k+1}+\|(\Upsilon_1,\Upsilon_2)\|^{2k+1}_{\mathcal{X}^{b, s}_{\alpha, T}}+\|(\phi_0, \varphi_0)\|^{2}_{\mathcal{H}^{ s }} \right),
\end{eqnarray*}
Therefore, for $(\Psi_{1.1},\Psi_{2.1}),(\Psi_{1.2},\Psi_{2.2})\in \mathcal{B}_{\mathcal{R},T} $ , we get
\begin{eqnarray*}
	\|\varGamma_1(\Psi_{1.1}-\Psi_{1.2})\|_{X^{b, s}_{T}}&\lesssim& T^{1-b+b^{\prime}}\left(\mathcal{R}^{2k}+\|(\Upsilon_1,\Upsilon_2)\|^{2k}_{\mathcal{X}^{b, s}_{\alpha, T}}+\|(\phi_0, \varphi_0)\|_{\mathcal{H}^{ s }} \right)\\ \\ &&\times\|(\Psi_{1.1}-\Psi_{1.2},\Psi_{2.1}-\Psi_{2.2})\|_{\mathcal{X}^{b, s}_{\alpha, T}},\\ \\ \|\varGamma_2(\Psi_{2.1}-\Psi_{2.2})\|_{X^{b, s}_{\alpha, T}}&\lesssim& T^{1-b+b^{\prime}}\left(\mathcal{R}^{2k}+\|(\Upsilon_1,\Upsilon_2)\|^{2k}_{\mathcal{X}^{b, s}_{\alpha, T}}+\|(\phi_0, \varphi_0)\|_{\mathcal{H}^{ s }} \right)\\ \\ &&\times\|(\Psi_{1.1}-\Psi_{1.2},\Psi_{2.1}-\Psi_{2.2})\|_{\mathcal{X}^{b, s}_{\alpha, T}},
\end{eqnarray*}
so
\begin{eqnarray*}
	\|(\varGamma_1(\Psi_{1.1}-\Psi_{1.2}),\varGamma_2(\Psi_{2.1}-\Psi_{2.2}))\|_{X^{b, s}_{\alpha, T}}&\lesssim& T^{1-b+b^{\prime}}\left(\mathcal{R}^{2k}+\|(\Upsilon_1,\Upsilon_2)\|^{2k}_{\mathcal{X}^{b, s}_{\alpha, T}}+\|(\phi_0, \varphi_0)\|_{\mathcal{H}^{ s }} \right)\\ \\ &&\times\|(\Psi_{1.1}-\Psi_{1.2},\Psi_{2.1}-\Psi_{2.2})\|_{\mathcal{X}^{b, s}_{\alpha, T}}.
\end{eqnarray*}

Let us choose $T_{\varrho }$ such that
\[4CT^{1-b+b^{\prime}}\left(\mathcal{R}^{2k}_{\varrho}+\|(\Upsilon_1,\Upsilon_2)\|^{2k}_{\mathcal{X}^{b, s}_{\alpha, T}}+\|(\phi_0, \varphi_0)\|_{\mathcal{H}^{ s }}  \right)\leq 1,\]

where $$\mathcal{R}_{\varrho }=2C\left(\|(\Upsilon_1,\Upsilon_2)\|^{2k+1}_{\mathcal{X}^{b, s}_{\alpha, T}}+\|(\phi_0, \varphi_0)\|^{2}_{\mathcal{H}^{ s }} \right).$$ One can easily verify that $(\varGamma_1,\varGamma_2)$ maps $\mathcal{B}_{\mathcal{R},T}$ into itself and is a strict contraction on $\mathcal{B}_{\mathcal{R},T}$ with respect to the norm $\|(\Psi_{1},\Psi_{2})\|_{\mathcal{X}^{b, s}_{\alpha, T}}$,

\begin{eqnarray*}
	\|(\varGamma_1(\Psi_{1.1}-\Psi_{1.2}),\varGamma_2(\Psi_{2.1}-\Psi_{2.2}))\|_{X^{b, s}_{\alpha, T}}&\leq& \dfrac{1}{4}\|(\Psi_{1.1}-\Psi_{1.2},\Psi_{2.1}-\Psi_{2.2})\|_{\mathcal{X}^{b, s}_{\alpha, T}}.
\end{eqnarray*}

Hence, $(\varGamma_1,\varGamma_2)$ has a unique fixed point in $\mathcal{X}^{b, s}_{\alpha, T}$, which is a solution of (\ref{sys2s}) on $[0,T_{\varrho}]$. \\
Now observe that 

\begin{equation}
	\left\{
	\begin{array}{ll}
		\mathfrak{U}(t)=\varPi_1(t)+\Psi_1(t)+\Upsilon_1(t)\in X^{b', s}_{T_{\varrho}}+X^{b, s}_{T_{\varrho}}\\ \\
		\mathfrak{U}_\alpha(t)=\varPi_2(t)+\Psi_2(t)+\Upsilon_2(t)\in X^{b', s}_{\alpha, T_{\varrho}}+X^{b, s}_{\alpha, T_{\varrho}}
	\end{array}
	\right.
\end{equation}
We complete the proof by showing that $(\phi,\varphi) \in C([ 0, T_{\varrho}] , H^{s}(\mathbb{R}))\times  C([ 0, T_{\varrho}] , H^{s}(\mathbb{R}))$. Recall that $b' < \frac{1}{2}$ and $b> \frac{1}{2}$. By  the Sobolev embedding theorem, we have $(\varPi_1,\varPi_2) \in C([ 0, T_{\varrho}] , H^{s}(\mathbb{R}))\times  C([ 0, T_{\varrho}] , H^{s}(\mathbb{R}))$. \\
We need the following theorem
\begin{theorem}[\cite{ref5}]\label{the5}
	Assume that $\mathcal{A}$ generates a contraction semigroup and $\Xi \in \mathcal{N}_\mathcal{W}^2\left([0, T], L_2^{0,s}\right)$. Then the process $\mathcal{W}_\mathcal{A}^{\Xi}(\cdot)$ has a continuous modification and there exists a constant $C$ such that
	$$
	\mathbb{E}\left(\sup_{\gamma \in[0, t]}{ }\left|\mathcal{W}_\mathcal{A}^{\Xi}(\gamma)\right|^2 \right) \leq C \mathbb{E}\left(\int_0^t\|\Xi(\gamma)\|_{L_2^0}^2 d \gamma \right) , \quad t \in[0, T] .
	$$	
	
\end{theorem}

Under the condition that $\Xi\in L^{0,s}_2 $ and the  fact that $\mathfrak{U}(t)$ and  $\mathfrak{U}_{\alpha}(t)$ are a unitary group in $H^{s}(\mathbb{R})$, an application of the Theorem \ref{the5}  implies  that $(\Upsilon_1,\Upsilon_2) \in C([ 0, T_{\varrho}] , H^{s}(\mathbb{R}))\times  C([ 0, T_{\varrho}] , H^{s}(\mathbb{R}))$.\\
By Lemma \ref{tri}, we have $\partial_x (N_\phi(\tilde{\phi}, \tilde{\varphi}))\in X^{s,b'}$ and $\partial_x (G(\tilde{\phi}, \tilde{\varphi})) \in X_{\alpha}^{s,b'}$ for any extension $\tilde{\phi}$ of $\phi$ in $X^{s,b'}+X^{s,b}$ and $\tilde{\varphi}$ of $\varphi$ in $X_{\alpha}^{s,b'}+X_{\alpha}^{s,b}$, where $-\frac{1}{2}<b^{\prime} \leq 0 \leq b<b^{\prime}+1$, and applying Lemma 3.2 in \cite{13}, we get
$$
\left\|\int_0^t \mathfrak{U}(t-\gamma)\partial_x(N_\phi(\tilde{\phi}, \tilde{\varphi}))d\gamma\right\|_{X^{s, b}} \leq C\left\|\partial_x(N_\phi(\tilde{\phi}, \tilde{\varphi}))\right\|_{X^{s,b'}} .
$$
$$
\left\|\int_0^t \mathfrak{U}_\alpha(t-\gamma)\partial_x (N_\varphi(\tilde{\phi}, \tilde{\varphi}))d\gamma\right\|_{X_{\alpha}^{s, b}} \leq C\left\|\partial_x(N_\varphi(\tilde{\phi}, \tilde{\varphi}))\right\|_{X_{\alpha}^{s,b'}} .
$$

Since $1+b'>\frac{1}{2}$, then $$(\tilde{\phi},\tilde{\varphi}) \in \mathcal{X}^{1+b', s}_{\alpha} \subset C([ 0, T_{\varrho}] , H^{s}(\mathbb{R}))\times  C([ 0, T_{\varrho}] , H^{s}(\mathbb{R})).$$

\end{document}